\documentclass[5p,twocolumn,times,authoryear,a4paper]{autart}     
\usepackage[utf8]{inputenc}
\usepackage{natbib}
\usepackage{amsmath}
\usepackage[export]{adjustbox}
\usepackage{mathtools}
\usepackage{amssymb}
\usepackage{color,soul}
\usepackage{enumitem}
\usepackage{comment}
\usepackage{graphicx}
\usepackage{tikz}
\usepackage{tkz-graph}
\newtheorem{theorem}{Theorem}   
\newtheorem{lemma}[theorem]{Lemma}
\newtheorem{problem}{Problem}
\newtheorem{proposition}[theorem]{Proposition}
\newtheorem{corollary}[theorem]{Corollary}
\newtheorem{example}{Example}
\newtheorem{remark}{Remark}

\newtheorem{dfn}{Definition}

\newcommand{\norm}[1]{\left\lVert#1\right\rVert}
\newenvironment{proof}{\textbf{Proof:}}{\hfill$\square$}
\DeclareMathOperator{\rank}{rank}
\DeclareMathOperator{\cker}{cker}

\DeclareMathOperator{\col}{col}

\DeclareMathOperator{\vect}{vec}
\DeclareMathOperator{\modulo}{mod}      

\begin{document}

\begin{frontmatter}

\title{Topology Identification of Heterogeneous Networks: Identifiability and Reconstruction} 


\author[add1,add2]{Henk J. van Waarde}\ead{h.j.van.waarde@rug.nl},
\author[add1,add3]{Pietro Tesi}\ead{pietro.tesi@unifi.it},        
\author[add2]{M. Kanat Camlibel}\ead{m.k.camlibel@rug.nl}  

\address[add1]{Engineering and Technology Institute Groningen, University of Groningen, 9747 AG Groningen, The Netherlands.}   
\address[add2]{Bernoulli Institute for Mathematics, Computer Science and Artificial Intelligence, University of Groningen, 9747 AG Groningen, The Netherlands.}  
\address[add3]{Department of Information Engineering, University of Florence, 50139 Florence, Italy.}

                                         
\begin{abstract}                          
This paper addresses the problem of identifying the graph structure of a dynamical network using measured input/output data. This problem is known as topology identification and has received considerable attention in recent literature. Most existing literature focuses on topology identification for networks with node dynamics modeled by single integrators or single-input single-output (SISO) systems. The goal of the current paper is to identify the topology of a more general class of heterogeneous networks, in which the dynamics of the nodes are modeled by general (possibly distinct) linear systems. Our two main contributions are the following. First, we establish conditions for topological identifiability, i.e., conditions under which the network topology can be uniquely reconstructed from measured data. We also specialize our results to homogeneous networks of SISO systems and we will see that such networks have quite particular identifiability properties. Secondly, we develop a topology identification method that reconstructs the network topology from input/output data. The solution of a generalized Sylvester equation will play an important role in our identification scheme.
\end{abstract}

\end{frontmatter}

\section{Introduction}
Graph structure plays an important role in the overall behavior of dynamical networks. Indeed, it is well-known that the convergence rate of consensus algorithms depends on the connectivity of the network topology. In addition, many properties of dynamical networks, like controllability, can be assessed on the basis of the network graph \cite{Liu2011,Chapman2013,Jia2019}. Unfortunately, the graph structure of dynamical networks is often unknown. This problem is particularly apparent in biology, for example in neural networks and genetic networks \cite{Julius2009}, but also emerges in other areas such as power grids \cite{Cavraro2018}. 

To deal with this problem, several topology identification methods have been developed. Such methods aim at reconstructing the topology (and weights) of a dynamical network on the basis of measured data obtained from the network. 

The paper \cite{Goncalves2008} studies necessary and sufficient conditions for dynamical structure reconstruction, see also \cite{Yuan2011}. A node-knockout scheme for topology identification was introduced in \cite{Nabi-Abdolyousefi2010} and further investigated in \cite{Suzuki2013}. Moreover, the paper \cite{Sanandaji2011} studies topology identification using compressed sensing, while \cite{Materassi2012} considers network reconstruction using Wiener filtering. A distributed algorithm for network reconstruction has also been studied \cite{Morbidi2014}. \cite{Shahrampour2015} study topology identification using power spectral analysis. In \cite{vanWaarde2019}, the network topology was reconstructed by solving certain Lyapunov equations. A Bayesian approach to the network identification problem was investigated in \cite{Chiuso2012}. The network topology was inferred from multiple independent observations of consensus dynamics in \cite{Segarra2017}. 
The paper \cite{Coutino2020} studies topology identification via subspace methods. There are also several results for topology reconstruction of nonlinear systems, see e.g., \cite{Wang2011,Timme2014,Shen2017}
albeit in this case few guarantees on the accuracy of identification
can be given.
In addition, we remark that the complementary problem of identifying the nodes dynamics assuming a \emph{known} topology has also been studied, see e.g. \cite{vandenHof2013,Haber2014,Hendrickx2019,vanWaarde2018,vanWaarde2018b,Ramaswamy2018,Cheng2019}, along with the joint topology and dynamics recovery problem
\cite{Ioannidis2019,Wai2019}. %

The goal of this paper is to provide a comprehensive treatment of topology identification for linear MIMO heterogeneous networks, with no assumptions on the network structure such as sparsity or regularity. Most existing work on topology identification emphasizes the role of the network topology by considering relatively simple node dynamics. For example, networks of single integrators have been studied in \cite{Nabi-Abdolyousefi2010,Morbidi2014,Hassan-Moghaddam2016,vanWaarde2019}. In addition, the papers \cite{Suzuki2013} and \cite{Shahrampour2015} consider homogeneous networks comprised of identical single-input single-output systems. Nonetheless, there are many examples of networks in which the subsystems are not necessarily the same, for example, mass-spring-damper networks \cite{Koerts2017}, where the masses at the nodes can be distinct. Heterogeneity in the node dynamics has also been studied in the detail in synchronization problems, see e.g. \cite{Wieland2011,Yang2014}.

We study topology identification for the general class of \emph{heterogeneous} networks, where the node dynamics are modelled by general, possibly distinct, MIMO linear systems. We divide our analysis in two parts, namely the study of \emph{identifiability} and the development of \emph{identification algorithms}. The study of identifiability of the network topology deals with the question whether there exists a data set from which the topology can be uniquely identified. Identifiability of the topology is hence a property of the node systems and the network graph, and is \emph{independent} of any data. Topological identifiability is an important property. Indeed, if it is not satisfied, then it is impossible to uniquely identify the network topology, regardless of the amount and richness of the data. After studying topological identifiability, we will turn our attention towards identification algorithms. Our two main contributions are hence the following:
\begin{enumerate}
	\item We provide conditions for topological identifiability of general heterogeneous networks. Our results recover an identifiability result for the special case of networks of single integrators \cite{Pare2013,vanWaarde2019}. We will also see that homogeneous networks of single-input single-output systems have quite special identifiability properties that do not extend to the general case of heterogeneous networks.  
	\item We establish a topology identification scheme for heterogeneous networks. The idea of the method is to reconstruct the interconnection matrix of the network by solving a \emph{generalized Sylvester equation} involving the Markov parameters of the network. We prove that the network topology can be uniquely reconstructed in this way, under the assumptions of topological identifiability and persistency of excitation of the input data.  
\end{enumerate}

A preliminary version of our work was presented in \cite{vanWaarde2019b}. The contributions of the current paper are significant in comparison to \cite{vanWaarde2019b} for two reasons. First, the identifiability results presented here are more general as they are applicable in situations when not all network nodes are excited. Also, the necessary conditions for identifiabi\-lity of single-integrator networks are shown to carry over to the more general class of homogeneous networks of single-input single-output systems. Secondly, the topology identification approach is new, and attractive in comparison to \cite{vanWaarde2019b} since the network interconnection matrix is computed directly and without the use of auxiliary variables. Our approach is also suitable for ``paralle\-lization" in the sense that each row block of the interconnection matrix can be computed independently.

The paper is organized as follows. In Section \ref{sectionproblem} we formulate the problem. Section \ref{sectionidentifiability} contains our results on topological identifiability. Subsequently, we describe our topology identification method in Section \ref{sectionresultsidentification}. Finally, we state our conclusions in Section \ref{sectionconclusion}. 

\subsection*{Notation}
We denote the \emph{Kronecker product} by $\otimes$. The \emph{direct sum} of matrices $A_1,A_2,\dots,A_k$ is the block diagonal matrix defined by 
$$
\bigoplus_{i = 1}^k A_i := \begin{pmatrix}
A_1 & 0 & \cdots & 0 \\
0 & A_2 & \cdots & 0 \\
\vdots & \vdots & \ddots & \vdots \\
0 & 0 & \cdots & A_k
\end{pmatrix}.
$$
Moreover, the \emph{concatenation} of matrices $A_1,A_2\dots,A_k$ of compatible dimensions is defined by
$$\col(A_1,A_2,\dots,A_k) := \begin{pmatrix}
A_1^\top & A_2^\top & \cdots & A_k^\top 
\end{pmatrix}^\top.$$
Finally, let $A(z)$ be an $n \times m$ rational matrix. Then the \emph{constant kernel} of $A(z)$ is $\cker A(z) := \{w \in \mathbb{R}^m \mid A(z)w = 0 \}$. 

\section{Problem formulation}
\label{sectionproblem}

We consider a network model similar to the one studied by Fuhrmann and Helmke \cite[Ch. 9]{Fuhrmann2015}. Specifically, we consider networks composed of $N$ discrete-time systems of the form
\begin{equation}
\label{nodesystems}
\begin{aligned}
x_i(t+1) &= A_i x_i(t) + B_i v_i(t) \\
w_i(t) &= C_i x_i(t),
\end{aligned}
\end{equation}
where $x_i(t) \in \mathbb{R}^{n_i}$ is the state of the $i$-th node system, $v_i(t) \in \mathbb{R}^{m_i}$ is its input and $w_i(t) \in \mathbb{R}^{p_i}$ is its output for $i = 1,2,\dots,N$. The real matrices $A_i$, $B_i$ and $C_i$ are of appropriate dimensions. We occasionally use the shorthand notation $(A_i,B_i,C_i)$ to denote \eqref{nodesystems}. The coupling between nodes is realized by the inputs $v_i(t)$, which are specified as
\begin{equation*}
v_i(t) = \sum_{j = 1}^N Q_{ij} w_j(t) + R_i u(t),
\end{equation*}
where $u(t) \in \mathbb{R}^m$ is the external network input and $Q_{ij}$ and $R_i$ are real matrices of appropriate dimensions. In addition, let $S_i$ be a real $p \times p_i$ matrix and consider the external network output $y(t) \in \mathbb{R}^p$, defined by
\begin{equation*}
y(t) = \sum_{i = 1}^N S_i w_i(t).
\end{equation*}
Then, by introducing the block diagonal matrices 
\begin{equation}
\label{ABC}
A = \bigoplus_{i = 1}^N A_i, \: B = \bigoplus_{i = 1}^N B_i, \text{ and }
C = \bigoplus_{i = 1}^N C_i,
\end{equation} 
and the matrices
\begin{align*}
Q &= \begin{pmatrix}
Q_{11} & \cdots & Q_{1N} \\
\vdots & \ddots & \vdots \\
Q_{N1} & \cdots & Q_{NN}
\end{pmatrix}, \:
R = \begin{pmatrix}
R_{1} \\
\vdots \\
R_{N} 
\end{pmatrix}, 
\: S^\top = \begin{pmatrix}
S_{1}^\top \\
\vdots \\
S_{N}^\top 
\end{pmatrix},
\end{align*}
we can represent the network dynamics compactly as
\begin{equation}
\label{system}
\begin{aligned}
x(t+1) &= (A + B Q C) x(t) + B R u(t) \\
y(t) &= S C x(t).
\end{aligned}
\end{equation}
Here $x(t) = \col(x_1(t),x_2(t),\dots,x_N(t)) \in \mathbb{R}^n$ where $n$ is defined as $n := \sum_{i = 1}^N n_i$. We emphasize that the coupling of the node dynamics is induced by the matrix $Q$, which we will hence call the \emph{interconnection matrix}.

There are a few important special cases of node dynamics \eqref{nodesystems} and resulting network dynamics \eqref{system}. If $A_i = A_0$, $B_i = B_0$ and $C_i = C_0$ for all $i = 1,2,\dots,N$, the dynamics of all nodes in the network are the same and the resulting dynamical network is called \emph{homogeneous}. The more general setting in which the node dynamics are not necessarily the same is referred to as a \emph{heterogeneous} network. Another special case of node dynamics occurs when $m_i = p_i = 1$ for all $i = 1,2,\dots,N$. In this case, the node systems are single-input single-output (SISO) systems, and the resulting dynamical network is referred to as a \emph{SISO network}\footnote{Here we emphasize that `SISO' refers to the node systems of the network. The overall network dynamics \eqref{system} can still have multiple external inputs and outputs.}. Topology identification of homogeneous SISO networks has been studied in \cite{Suzuki2013} and \cite{Shahrampour2015}. In addition, topology identification has been well-studied (see e.g \cite{Goncalves2008,Nabi-Abdolyousefi2010,Hassan-Moghaddam2016,vanWaarde2019}) for networks of so-called \emph{single-integrators}, in which the node dynamics are described by $\dot{x}_i(t) = v_i(t)$. This type of node dynamics can be seen continuous-time counterpart of \eqref{nodesystems} where $A_i = 0$, $B_i = 1$ and $C_i = 1$ for $i = 1,2,\dots,N$.

The purpose of this paper is to study topology identification for general, heterogeneous dynamical networks of the form \eqref{system}. Although we focus on discrete-time systems, our results can be stated for continuous-time systems as well. In order to make the problem more precise, we first explain what we mean by the topology of \eqref{system}. Let $\mathcal{G} = (\mathcal{V},\mathcal{E})$ be a weighted directed graph with $\mathcal{V} = \{1,2,\dots,N\}$ and $\mathcal{E} \subseteq \mathcal{V} \times \mathcal{V}$ such that $(j,i) \in \mathcal{E}$ if and only if $Q_{ij} \neq 0$. Each edge $(j,i) \in \mathcal{E}$ is weighted by the nonzero matrix $Q_{ij}$. We refer to $\mathcal{G}$ as the \emph{topology} of the dynamical network \eqref{system}. With this in mind, the problem of \emph{topology identification} concerns finding $\mathcal{G}$ (equiva\-lently, finding $Q$) using measurements of the input $u(t)$ and output $y(t)$ of \eqref{system}. We assume knowledge of the local node dynamics (i.e., the matrices $A,B$ and $C$) as well as the external input/output matrices $R$ and $S$\footnote{This assumption is standard in the literature on topology identification, see, e.g., \cite{Shahrampour2015} and \cite{Suzuki2013}. Without knowledge of the node dynamics, topology identification becomes a full system identification problem.}. 

At this point, we may ask the following natural question: is it possible to \emph{uniquely} reconstruct the topology of \eqref{system} from input/output data? To formalize and answer this question, we define the notion of \emph{topological identifiability}. Let $y_{u,x_0,Q}(t)$ denote the output of \eqref{system} at time $t$, where the subscript emphasizes the dependence on the input $u( \cdot )$, the initial condition $x_0 = x(0)$ and interconnection matrix $Q$. The following definition is inspired by \cite{Grewal1976} and defines the notion of \emph{distinguishability} of interconnection matrices. 

\begin{dfn}
	\label{definitiondistinguishability}
	Let $y_{u,x_0,Q}( \cdot )$ and $y_{u,\bar{x}_0,\bar{Q}}( \cdot )$ denote the output trajectories of two systems of the form \eqref{system} with interconnection matrices $Q$ and $\bar{Q}$ and initial conditions $x_0$ and $\bar{x}_0$, respectively. We say that $Q$ and $\bar{Q}$ are \emph{indistinguishable} if there exist initial conditions $x_0, \bar{x}_0 \in \mathbb{R}^n$ such that  
	\begin{equation*}
	y_{u,x_0,Q}( \cdot ) = y_{u,\bar{x}_0,\bar{Q}}( \cdot )
	\end{equation*}
	for all input functions $u$. Moreover, $Q$ and $\bar{Q}$ are said to be \emph{distinguishable} if they are not indistinguishable. 
\end{dfn}

With this in mind, the topology of \eqref{system} is said to be identifiable if $Q$ is distinguishable from all other interconnection matrices. More formally, we have the following definition. 

\begin{dfn}
	Consider system \eqref{system} with interconnection matrix $Q$. The topology of system \eqref{system} is said to be \emph{identifiable} if $Q$ and $\bar{Q}$ are distinguishable for all real $\bar{Q} \neq Q$.
\end{dfn}

The importance of topological identifiability lies in the fact that unique reconstruction of $Q$ from input/output data is \emph{only} possible if the topology of \eqref{system} is identifiable. Indeed, if this is not the case, there exists some $\bar{Q} \neq Q$ that is indistinguishable from $Q$, meaning that both $Q$ and $\bar{Q}$ explain \emph{any} input/output trajectory of \eqref{system}. Topological identifiability is hence a structural property of the system \eqref{system} that is independent of a particular data sequence and that is \emph{necessary} for the unique reconstruction of $Q$ from data. 

Following \cite{Grewal1976}, it is straightforward to characterize topological identifiability in terms of the transfer matrix from $u$ to $y$. This transfer function will be denoted by
\begin{equation}
\label{networktransfer}
F_Q(z) := SC(zI-A-BQC)^{-1}BR.
\end{equation}
\begin{proposition}
	\label{theoremtransfer}
	The topology of the networked system \eqref{system} is identifiable if and only if the following implication holds:
	\begin{equation*}
	F_Q(z) = F_{\bar{Q}}(z) \text{ for real } \bar{Q} \implies Q = \bar{Q}.	\end{equation*}
\end{proposition}

Although Proposition \ref{theoremtransfer} provides a necessary and sufficient condition for topological identifiability, the condition involves the arbitrary matrix $\bar{Q}$. Hence, it is not clear how to verify the condition of Proposition \ref{theoremtransfer}. Instead, in this paper we want to establish conditions for topological identifiability in terms of the local system matrices $A$, $B$ and $C$ and the matrices $Q$, $R$ and $S$. This is formalized in the following problem. 
\begin{problem}
	\label{problem1}
	Find necessary and sufficient conditions on the node dynamics $A$, $B$, $C$, the external input/output matrices $R$, $S$ and the interconnection matrix $Q$ under which the topology of \eqref{system} is identifiable. 
\end{problem}

Our second goal is to identify $Q$ from input/output data.
\begin{problem}
	\label{problem2}	
	Develop a methodology to identify the interconnection matrix $Q$ from measurements of the input $u(t)$ and output $y(t)$ of system \eqref{system}.
\end{problem}

\section{Conditions for topological identifiability}
\label{sectionidentifiability}

In this section we state our solution to Problem \ref{problem1} by providing necessary and sufficient conditions for topological identifiability. We start by providing an overview of the results that are proven in this section. In the following table, ``\textbf{N}" denotes necessary 
and ``\textbf{S}" denotes sufficient.
\begin{center}
	\begin{tabular}{|c|c|}
		\hline 
		\small
		Thm. \ref{necsufconditions} & \small General \textbf{N}-\textbf{S} conditions \\
		\hline
		\small
		Thm. \ref{SRfullrank} & 
		\small \textbf{N} condition; also \textbf{S} if $R$ has full rank \\
		\hline
		\small
		Thm. \ref{theoremfullrank} & 
		\small \textbf{N} condition for homogeneous SISO networks \\
		\hline
		\small
		Thm. \ref{theoremSISO} & 
		\small \textbf{N}-\textbf{S} conditions for homog. SISO networks \\
		\hline
	\end{tabular}
\end{center}

For analysis purposes, we first rewrite the network transfer matrix $F_Q(z)$. Note that
\begin{equation*}
zI-A = (zI - A - BQC) + BQC.
\end{equation*}
Premultiplication by $(zI-A)^{-1}$ and postmultiplication by the matrix $(zI - A - B Q C)^{-1}$ yields
\begin{align*}
&(zI - A - B Q C)^{-1} = \\ &(zI-A)^{-1} + (zI-A)^{-1} B Q C (zI - A - B Q C)^{-1}.
\end{align*}
This means that
\begin{align*}
&C (zI - A - B Q C)^{-1} B = \\ &G(z) + G(z) Q C (zI - A - B Q C)^{-1} B,
\end{align*}
where $G(z) = C (zI-A)^{-1}B$ is a block diagonal matrix containing the transfer matrices of all node systems. Finally, by rearranging terms we obtain 
\begin{equation}
\label{relationI-GQ}
C (zI - A - B Q C)^{-1} B = \left( I-G(z)Q \right)^{-1} G(z).
\end{equation}
Note that the inverse of $I-G(z)Q$ exists as a rational matrix. Indeed, since $(zI-A)^{-1}$ is strictly proper we see that $\lim_{z \to \infty} (I-G(z)Q) = I$. Therefore, we conclude by \eqref{relationI-GQ} that the transfer matrix $F_Q(z)$ equals
\begin{equation}
\label{networktransfer2}
F_Q(z) = S\left( I-G(z)Q \right)^{-1} G(z)R.
\end{equation}
We remark that \eqref{networktransfer2} is an attractive representation of the network transfer matrix, since the matrices $A$, $B$ and $C$ describing the local system dynamics are grouped and contained in the transfer matrix $G(z)$.

\begin{remark}
	By \eqref{networktransfer2}, we see that the networked system \eqref{system} can be represented by the block diagram in Figure \ref{fig:blockdiagram}. Hence, the problem of topology identification can be viewed as the identification of the \emph{static output feedback gain} $Q$, assuming knowledge of the system $G(z)$ and the external input/output matrices $R$ and $S$.
	\begin{figure}[h!]
		\centering
		\includegraphics[width=.45\textwidth]{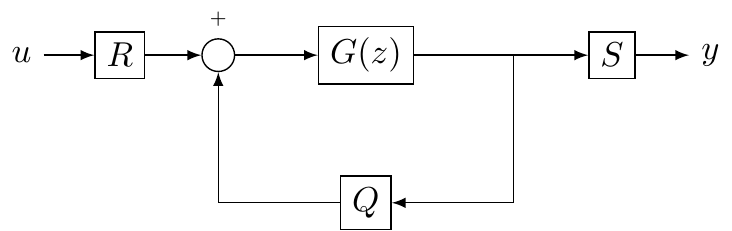}
		\label{fig:blockdiagram}
		\caption{Block diagram of the networked system \eqref{system}.}
	\end{figure}
\end{remark}

The following theorem gives necessary and sufficient conditions for topological identifiability. We will use the notation $G_i(z) := C_i (zI-A_i)^{-1}B_i$ to denote the transfer matrix from $v_i$ to $w_i$ of node system $i\in \mathcal{V}$.
\begin{theorem}
	\label{necsufconditions}
	Consider the networked system \eqref{system} and assume that the matrix $S$ has full column rank. The topology of \eqref{system} is identifiable if and only if
	\begin{equation}
	\label{ckernel}
	\cker \left( G_i(z) \otimes H_Q^\top(z) \right) = \{0\} \text{ for all } i \in \mathcal{V},
	\end{equation}
	where $H_Q(z) := \left( I-G(z)Q \right)^{-1} G(z)R$.
\end{theorem}

\begin{proof}
	Suppose that $F_Q(z)=F_{\bar{Q}}(z)$, where $\bar{Q}$ is real. Then, from \eqref{networktransfer2} we have
	\begin{equation*}
	S\left( I-G(z)Q \right)^{-1} G(z) R = S\left( I-G(z)\bar{Q} \right)^{-1} G(z)R.
	\end{equation*}
	By hypothesis, $S$ has full column rank and hence
	\begin{equation}
	\label{eqQQbar}
	\left( I-G(z)Q \right)^{-1} G(z) R = \left( I-G(z)\bar{Q} \right)^{-1} G(z)R.
	\end{equation}
	
	We define $\Delta := Q-\bar{Q}$. Then, \eqref{eqQQbar} is equivalent to each of the following statements:
	\begin{align*}
	\left( I-G(z)\bar{Q} \right) \left( I-G(z)Q \right)^{-1} G(z) R &= G(z) R \\
	\left( I-G(z)(Q - \Delta) \right) \left( I-G(z)Q \right)^{-1} G(z) R &= G(z) R \\
	G(z) \Delta \left( I-G(z)Q \right)^{-1} G(z)R &= 0 \\
	G(z) \Delta H_Q(z) &= 0.
	\end{align*}
	Equivalently, 
	\begin{equation}
	\label{eqGDeltaG} 
	H_Q^\top(z) \Delta^\top G^\top(z) = 0.
	\end{equation}
	Next, let $\vect(M)$ denote the vectorization of a matrix $M$. Then \eqref{eqGDeltaG} is equivalent to
	\begin{equation}
	\label{eqvect}
	(G(z) \otimes H_Q^\top(z)) \vect(\Delta^\top) = 0.
	\end{equation}
	By \eqref{eqvect} it is clear that the topology of \eqref{system} is identifiable if and only if the constant kernel of $G(z) \otimes H_Q^\top(z)$ is zero. Finally, by the block diagonal structure of $G(z)$, this is equivalent to \eqref{ckernel} which proves the theorem.
\end{proof}

By Theorem \ref{necsufconditions}, topological identifiability is equivalent to the matrices $G_i(z) \otimes H_Q^\top(z)$ having zero constant kernel. Note that this condition generally depends on the -a priori unknown- matrix $Q$. Notably, identifiability is \emph{independent} of the particular matrix $Q$ whenever all node inputs are excited and all node outputs are measured, as stated in the following theorem. 

\begin{theorem}
	\label{SRfullrank}
	Consider the networked system \eqref{system}. If the topology of \eqref{system} is identifiable then
	\begin{equation}
	\label{constantkernel}
	\cker \left(G_i^\top(z) \otimes G_j(z)\right) = \{0\}
	\end{equation}
	for all $i,j \in \mathcal{V}$. In addition, suppose that $S$ has full column rank and $R$ has full row rank. Then the topology of \eqref{system} is identifiable \emph{if and only if} \eqref{constantkernel} holds. 
\end{theorem}

The importance of Theorem \ref{SRfullrank} lies in the fact that the identifiability condition \eqref{constantkernel} can be verified \emph{without knowledge} of $Q$. This means that, whenever the rank conditions on $S$ and $R$ hold, one can check for topological identifiability before collecting data from the system.  %

\begin{remark}
	A proper transfer matrix $T(z)$ has constant kernel $\{0\}$ if and only if the matrix $\col(M_0,M_1,\dots,M_r)$ has full column rank. Here $M_0,M_1,\dots,M_r$ are the Markov parameters of $T(z)$ and $r$ is greater or equal to the order of $T(z)$. As such, the conditions of Theorems \ref{necsufconditions} and \ref{SRfullrank} can be verified by computing the rank of the Markov parameter matrices associated to the transfer matrices in \eqref{ckernel} and \eqref{constantkernel}.
\end{remark}

\begin{proof}
	We first prove the second statement. Suppose that $S$ has full column rank and $R$ has full row rank. Then $F_Q(z) = F_{\bar{Q}}(z)$ is equivalent to 
	\begin{equation*}
	\left( I-G(z)Q \right)^{-1} G(z) = \left( I-G(z)\bar{Q} \right)^{-1} G(z).
	\end{equation*}
	We define $\Delta := Q - \bar{Q}$. Then, $F_Q(z) = F_{\bar{Q}}(z)$ is equivalent to
	$$
	G(z)\Delta (I-G(z)Q)^{-1}G(z) = 0,
	$$
	In other words, $G(z)\Delta G(z) (I-QG(z))^{-1} = 0$. This in turn is equivalent to $G(z) \Delta G(z) = 0$. In other words, $\left( G^\top(z) \otimes G(z) \right) \vect(\Delta) = 0$. Exploiting the block diagonal structure of $G(z)$, we conclude that the topology of \eqref{system} is identifiable if and only if \eqref{constantkernel} holds.
\end{proof}


A consequence of Theorem \ref{SRfullrank} is that identifiability of the topology of \eqref{system} implies that the constant kernel of both $G_i^\top(z)$ and $G_i(z)$ is zero for all $i \in \mathcal{V}$. Based on this fact, we relate topological identifiability and output controllability of the node systems. 
\begin{dfn}
	Consider the system
	\begin{equation}
	\begin{aligned}
	\label{linsys}
	x(t+1) &= A x(t) + B u(t) \\
	y(t) &= C x(t),
	\end{aligned}
	\end{equation}
	where $x \in \mathbb{R}^n$, $u \in \mathbb{R}^m$ and $y \in \mathbb{R}^p$, and let $y_{u,x_0}( \cdot )$ denote the output trajectory of \eqref{linsys} for a given initial condition $x_0$ and input $u( \cdot )$. System \eqref{linsys} is called \emph{output controllable} if for every $x_0 \in \mathbb{R}^{n}$ and $y_1 \in \mathbb{R}^{p}$ there exists an input $u( \cdot )$ and time instant $T \in \mathbb{N}$ such that $y_{x_0,u}(T) = y_1$. 
\end{dfn}

\begin{corollary}
	\label{coroutputcont}
	If the topology of \eqref{system} is identifiable then the systems $(A_i,B_i,C_i)$ and $(A_i^\top,C_i^\top,B_i^\top)$ are output controlla\-ble for all $i \in \mathcal{V}$.
\end{corollary}

\begin{proof}
	By Theorem \ref{SRfullrank}, identifiability of the topology of \eqref{system} implies that the constant kernel of $G_i^\top(z)$ is zero for all $i \in \mathcal{V}$. Now, for $w \in \mathbb{R}^{p_i}$ we have $w^\top G_i(z) = 0$ if and only if $w^\top C_i A_i^k B_i = 0$ for all $k = 0,1,\dots$, equivalently, $w^\top C_i A_i^k B_i = 0$ for all $k = 0,1,\dots,n_i-1$. Hence, 
	\begin{equation*}
	w^\top \begin{pmatrix}
	C_i B_i & C_i A_i B_i & \cdots & C_i A_i^{n-1} B_i
	\end{pmatrix} = 0 \implies w = 0.
	\end{equation*}
	The latter implication holds if and only if the output controllability matrix of $(A_i,B_i,C_i)$ has full row rank, equivalently $(A_i,B_i,C_i)$ is output controllable \cite[Ex. 3.22]{Trentelman2001}. The proof for the ne\-cessity of output controllability of $(A_i^\top,C_i^\top,B_i^\top)$ is analogous and hence omitted.
\end{proof}

\begin{remark}
	Output controllability of $(A_i,B_i,C_i)$ can be interpreted as an `excitability' condition. Indeed, it guarantees that we have enough freedom in steering the output $w_i(t)$ of each node $i \in \mathcal{V}$.
\end{remark}

\begin{example}
	\label{example1}
	We will now illustrate Theorems \ref{necsufconditions} and \ref{SRfullrank}. Consider a network of $N = 10$ oscillators of the form
	\begin{align*}
	x_i(t+1) &= \begin{pmatrix}
	\cos\theta_i & \sin\theta_i \\
	-\sin\theta_i & \cos\theta_i
	\end{pmatrix} x_i(t) + \begin{pmatrix}
	1 \\ 0
	\end{pmatrix}
	v_i(t) \\
	w_i(t) &= \begin{pmatrix}
	1 & 0
	\end{pmatrix}x_i(t),
	\end{align*}
	where $\theta_i \in \mathbb{R}$ is a constant, given by $\theta_i = (0.2+0.01 i)\pi$ for $i = 1,2,\dots,N$. The network topology is a \emph{cycle graph} $\mathcal{G} = (\mathcal{V},\mathcal{E})$ (with self-loops), defined by $\mathcal{V} := \{1,2,\dots,N\}$ and $\mathcal{E} := \{(i,j) \mid i-j \equiv -1,0,1 (\modulo N) \}$. Here $\modulo$ denotes the modulo operation and $\equiv$ denotes congruence. The network nodes are diffusively coupled, and an external input is applied to node 1, that is, 
	$$
	v_i(t) = \begin{cases}
	\frac{1}{2} \sum_{j \in \mathcal{N}_i} (w_j(t)-w_i(t)) + u(t) & \text{if } i = 1 \\
	\frac{1}{2} \sum_{j \in \mathcal{N}_i} (w_j(t)-w_i(t)) & \text{otherwise},
	\end{cases} 
	$$
	where $\mathcal{N}_i := \{ j \mid (j,i) \in \mathcal{E} \}$. This means that the interconnection matrix $Q$ is defined element-wise as
	$$
	Q_{ij} = \begin{cases}
	1 & \text{if } i = j \\
	-\frac{1}{2} & \text{if } i \neq j \text{ and } (j,i) \in \mathcal{E} \\
	0 & \text{otherwise.} 
	\end{cases}
	$$
	Since we only externally influence the first node system, the corresponding matrix $R$ is given by the first column of $I$. We assume that we externally measure all node outputs, meaning that $S = I$. 
	
	Using Theorem \ref{necsufconditions}, we want to show that the topology of \eqref{system} is identifiable. First, note that the transfer function $G_i(z)$ of node system $i$ is given by 
	$$
	G_i(z) = \frac{z - \cos \theta_i}{z^2 - 2 z \cos \theta_i +1},
	$$
	which is nonzero for all $i \in \mathcal{V}$. Since $G_i(z)$ is scalar, Theorem \ref{necsufconditions} implies that the topology of \eqref{system} is identifiable if and only if $\cker H^\top_Q(z) = \{0\}$. This is equivalent to the output controllability of the system $(A+BQC,BR,C)$. It can be easily verified that the output controllability matrix 
	$$
	\begin{pmatrix}
	CBR & C(A+BQC)BR & \cdots & C(A+BQC)^{N-1}BR
	\end{pmatrix}
	$$
	has full row rank. We therefore conclude by Theorem \ref{necsufconditions} that the topology of \eqref{system} is identifiable. Note that the rank of the output controllability matrix (and hence, identifiability) depends on the interconnection matrix $Q$. 
	
	Next, we discuss the scenario in which $R = I$. In this case, we can externally influence all nodes. Now, identifiability can be checked without knowledge of $Q$. In fact, by Theorem \ref{SRfullrank}, the topology of \eqref{system} is identifiable if and only if $\cker \left(G_i^\top(z) \otimes G_j(z)\right) = \{0\}$. This condition is satisfied, since all local transfer functions are nonzero scalars.
\end{example}

So far, we have provided a general condition for identifiability in Theorem \ref{necsufconditions}, and we have discussed some of the implications of this result in Theorem \ref{SRfullrank} and Corollary \ref{coroutputcont}. However, possible criticism of the results may arise from the full rank condition on $S$ in Theorem \ref{necsufconditions}, which, until now, has been left rather unjustified. 

It turns out that full column rank of $S$ (or the dual, full row rank of $R$) is \emph{necessary} for topological identifiability in case the networked system is homogeneous and SISO. For this important class of networked systems, the rank condition on $S$ in Theorem \ref{necsufconditions} is hence not restrictive. 

\begin{theorem}
	\label{theoremfullrank}
	Consider a homogeneous SISO network, that is, a system of the form \eqref{system} with $m_i = p_i = 1$ and $A_i = A_0$, $B_i = B_0$ and $C_i = C_0$ for all $i \in \mathcal{V}$. If the topology of \eqref{system} is identifiable then $\rank S = N$ or $\rank R = N$. 
\end{theorem}

\begin{remark}
	Theorem \ref{theoremfullrank} generalizes several known results (see \cite{Pare2013,vanWaarde2019,vanWaarde2019b}) for networks of single-integrators. Indeed, in the special case that $A_0 = 0$, $B_0 = C_0 = 1$, the node output $w_i(t)$ equals the node state $x_i(t)$ for all $i \in \mathcal{V}$, and Theorem \ref{theoremfullrank} asserts that either full state measurement or full state excitation is necessary for identifiability. This fact has been observed in different setups in \cite[Thm. 1]{Pare2013}, \cite[Rem. 2]{vanWaarde2019}, and \cite[Thm. 5]{vanWaarde2019b}. 
\end{remark}%

Before proving Theorem \ref{theoremfullrank}, we state the following lemma.


\begin{lemma}
	\label{lemmacont}
	Suppose that $m_i = p_i = 1$ and $A_i = A_0$, $B_i = B_0$ and $C_i = C_0$ for all $i \in \mathcal{V}$. If the topology of \eqref{system} is identifiable then $(Q,R)$ is controllable and $(S,Q)$ is observable. 
\end{lemma}

\begin{proof}
	Suppose on the contrary that $(S,Q)$ is unobser\-vable. Let $v \in \mathbb{R}^N$ be a nonzero vector in the unobservable subspace of $(S,Q)$, i.e.,
	\begin{equation*}
	SQ^k v = 0 \text{ for all } k \in \mathbb{N}.
	\end{equation*}
	This implies that $SQ^k = S(Q+vv^\top)^k$ for all $k \in \mathbb{N}$. By \eqref{networktransfer2}, the network transfer matrix is given by 
	$$F_Q(z) = S(I-G_0(z)Q)^{-1}G_0(z)R,$$ 
	where $G_0(z) := C_0(zI-A_0)^{-1}B_0$ is a scalar transfer function. Next, by expanding $F_Q(z)$ as a formal series 
	\begin{equation*}
	F_Q(z) = S \left(\sum_{k=0}^\infty (Q G_0(z))^k \right) G_0(z)R,
	\end{equation*}
	it is clear that $F_Q(z) = F_{\bar{Q}}(z)$, where the matrix $\bar{Q}$ is defined as $\bar{Q} := Q + vv^\top$. Since $v \neq 0$, the matrices $Q$ and $\bar{Q}$ are distinct. Hence, the topology of \eqref{system} is not identifiable. The proof for necessity of controllability of $(Q,R)$ is analogous and therefore omitted. 
\end{proof}

\textbf{Proof of Theorem \ref{theoremfullrank}}: Suppose on the contrary that $\rank R < N$ and $\rank S < N$. Then there exist nonzero vectors $v_1,v_2 \in \mathbb{R}^N$ such that $Sv_1 = 0$ and $v_2^\top R = 0$. We assume without loss of generality that $v_2$ is such that $v_2^\top v_1 \neq -1$. Next, we define $T := I + v_1v_2^\top$. By the Sherman-Morrison formula, $T$ is invertible if and only if $1+v_2^\top v_1 \neq 0$, equivalently, $v_2^\top v_1 \neq -1$. By our assumption on $v_2$, the matrix $T$ is hence invertible, and 
	\begin{equation*}
	T^{-1} = I - \frac{v_1 v_2^\top}{1+ v_2^\top v_1}.
	\end{equation*}
	We define the matrix 
	\begin{equation}
	\label{definitionQbar}
	\bar{Q} := T^{-1} Q T = \left( I - \frac{v_1 v_2^\top}{1+ v_2^\top v_1} \right) Q (I + v_1 v_2^\top).     
	\end{equation}
	Now, we distinguish two cases: $Q \neq \bar{Q}$ and $Q = \bar{Q}$. First suppose that $Q \neq \bar{Q}$. Since we have $\bar{Q} = T^{-1}Q T$, $TR = R$ and $ST^{-1} = S$, we obtain 
	\begin{align*}
	\mathcal{T}(I \otimes A_0 + Q \otimes B_0 C_0)\mathcal{T}^{-1} &= I \otimes A_0 + \bar{Q} \otimes B_0 C_0 \\
	\mathcal{T}(I \otimes B_0) R &= (I \otimes B_0)R \\
	S(I \otimes C_0)\mathcal{T}^{-1} &= S(I \otimes C_0),
	\end{align*}
	where $\mathcal{T} := T \otimes I$. Here we have used the fact that $p_i = m_i = 1$ for all $i \in \mathcal{V}$, as well as the property $(X_1 \otimes Y_1)(X_2 \otimes Y_2) = (X_1 X_2) \otimes (Y_1 Y_2)$ for matrices $X_1$, $X_2$, $Y_1$, $Y_2$ of compatible dimensions. We conclude that $F_Q(z) = F_{\bar{Q}}(z)$, i.e., the topology of \eqref{system} is not identifiable. 
	
	Secondly, suppose that $Q = \bar{Q}$. It follows from \eqref{definitionQbar} that
	\begin{equation*}
	Qv_1 v_2^\top - \frac{v_1 v_2^\top}{1+ v_2^\top v_1}Q - \frac{v_1 v_2^\top}{1+ v_2^\top v_1} Q v_1 v_2^\top = 0,
	\end{equation*}
	equivalently,
	\begin{equation*}
	(1+ v_2^\top v_1)Qv_1 v_2^\top - v_1 v_2^\top Q - v_1 v_2^\top Q v_1 v_2^\top = 0.
	\end{equation*}
	Multiply from right by $v_2$ and rearrange terms to obtain 
	\begin{equation*}
	(1+ v_2^\top v_1) v_2^\top v_2 Qv_1  = (v_2^\top Q v_2 + v_2^\top Q v_1 v_2^\top v_2) v_1.
	\end{equation*}
	This means that $v_1$ is an eigenvector of $Q$ contained in the kernel of $S$. Therefore, $(S,Q)$ is unobservable (cf. \cite[Ch. 3]{Trentelman2001}). By the previous lemma, this implies that the topology of \eqref{system} is not identifiable. \hfill$\square$

Theorem \ref{theoremfullrank} is interesting because it shows that the ability to measure all node outputs or to excite all node inputs is \emph{necessary} for identifiability in the case of homogeneous SISO networks. This result allows us to sharpen Theorem \ref{necsufconditions} for this particular class of networks.

\begin{theorem}
	\label{theoremSISO}
	Consider a homogeneous SISO network, that is, a system of the form \eqref{system} with $m_i = p_i = 1$ and $A_i = A_0$, $B_i = B_0$ and $C_i = C_0$ for all $i \in \mathcal{V}$. The topology of \eqref{system} is identifiable if and only if $G_0(z) := C_0 (zI - A_0)^{-1}B_0 \neq 0$ and at least one of the following two conditions holds:
	\begin{enumerate}[label=(\roman*)]
		\item $\rank S = N$ and $(Q,R)$ is controllable \label{cond1}
		\item $\rank R = N$ and $(S,Q)$ is observable. \label{cond2}
	\end{enumerate}
\end{theorem}

\begin{proof}
	To prove the `if'-statement, we first assume that $G_0(z)$ is nonzero, $\rank S = N$ and $(Q,R)$ is controllable. By Theorem \ref{necsufconditions}, the topology of \eqref{system} is identifiable if and only if $\cker H_Q^\top(z) = \{0\}$, where $H_Q(z)$ is given by $H_Q(z) = (I-G_0(z)Q)^{-1} G_0(z)R$. We expand the latter matrix as a formal series as
	\begin{equation}
	\label{series}
	(I-G_0(z)Q)^{-1} G_0(z)R = \left(\sum_{k=0}^\infty (G_0(z)Q)^k\right) G_0(z)R.
	\end{equation}
	We claim that by strict properness of $G_0(z)$, the powers $G_0^k(z)$ ($k = 0,1,2,\dots$) 
	are linearly independent over the reals. Indeed, suppose $\alpha_1 G_0^{k_1}(z) + \cdots + \alpha_r G_0^{k_r}(z) = 0$ for $\alpha_1,\dots,\alpha_r \in \mathbb{R}$ and $k_1 < \cdots < k_r$. Let $G_0(z) = \frac{p_0(z)}{q_0(z)}$ where $p_0$ and $q_0$ are polynomials. If $\alpha_1 \neq 0$ then 
	\begin{equation}
	\label{polynomials}
	\frac{p_0^{k_1}(z)q_0^{k_r-k_1}(z)}{q_0^{k_r}(z)} = - \frac{1}{\alpha_1} \sum_{i=2}^{r} \alpha_i \frac{p_0^{k_i}(z)q_0^{k_r-k_i}(z)}{q_0^{k_r}(z)}.
	\end{equation}
	By strict properness of $G_0(z)$, this is a contradiction since every numerator on the right hand side of \eqref{polynomials} has degree less than $p_0^{k_1}(z) q_0^{k_r-k_1}(z)$. Thus $\alpha_1 = 0$. In fact, we can repeat the same argument to show $\alpha_1 = \cdots = \alpha_r = 0$, proving the claim of independence. It follows from \eqref{series} that $v \in \mathbb{R}^N$ satisfies $v^\top H_Q(z) = 0$ if and only if 
	$$
	\sum_{k=0}^{\infty} G_0^k(z) v^\top Q^k R = 0, 
	$$
	where we leveraged the hypothesis that $G_0(z)$ is nonzero. Now, using the fact that $G_0^k(z)$ ($k=0,1,2,\dots$) are linearly independent, we obtain $v^\top Q^k R = 0$ for all $k \in \mathbb{N}$. We conclude by controllability of the pair $(Q,R)$ that $v = 0$, hence $\cker H_Q^\top(z) = \{0\}$. In other words, the topology of \eqref{system} is identifiable. The sufficiency of the three conditions $G_0(z) \neq 0$, $\rank R = N$ and $(S,Q)$ is observable is proven in a similar fashion and thus omitted.
	
	To prove the `only if'-statement, suppose that the topology of \eqref{system} is identifiable. Clearly, this implies that $G_0(z) \neq 0$. Indeed, if $G_0(z) = 0$ then $F_Q(z) = 0$ and any $\bar{Q}$ satisfies $F_Q(z) = F_{\bar{Q}}(z)$. By Lemma \ref{lemmacont}, $(Q,R)$ is controllable and $(S,Q)$ is observable. Furthermore, by Theorem \ref{theoremfullrank}, either $S$ or $R$ has full rank.
\end{proof}

It is noteworthy that full rank of either $R$ or $S$ is not necessary for topological identifiability of heterogeneous networks, as demonstrated next.

\begin{example}
	Consider a networked system \eqref{system} consisting of two nodes $A_1 = 0, B_1 = 1$, and $C_1 = 1$, and 
	\begin{align*}	
	A_2 = \begin{pmatrix}
	0 & 1 \\ 0 & 0
	\end{pmatrix}, \: B_2 = \begin{pmatrix}
	0 \\ 1
	\end{pmatrix}, \:
	C_2 = \begin{pmatrix}
	1 & 0
	\end{pmatrix}.
	\end{align*} 
	In addition, assume that $R = \begin{pmatrix}
	1 & 0
	\end{pmatrix}^\top$ and $S = \begin{pmatrix}
	0 & 1
	\end{pmatrix}$. It can be easily verified that
	\begin{equation*}
	F_Q(z) = \frac{Q_{21}}{z^3-Q_{11}z^2-Q_{22}z+Q_{11}Q_{22}-Q_{12}Q_{21}},
	\end{equation*}
	where $Q_{11}$, $Q_{12}$, $Q_{21}$ and $Q_{22}$ are the entries of the interconnection matrix 
	\begin{equation*}
	Q = \begin{pmatrix}
	Q_{11} & Q_{12} \\
	Q_{21} & Q_{22}
	\end{pmatrix}.
	\end{equation*}
	We assume that $Q_{21} \neq 0$ such that $F_Q(z)$ is nonzero. Suppose that $F_Q(z) = F_{\bar{Q}}(z)$ for some interconnection matrix $\bar{Q}$. By comparing the numerators of $F_Q$ and $F_{\bar{Q}}$ we see that $Q_{21} = \bar{Q}_{21}$. Moreover, by comparing the coefficients corresponding to $z^2$ and $z$ in the denominator, we obtain $Q_{11} = \bar{Q}_{11}$ and $Q_{22} = \bar{Q}_{22}$. Finally, by comparing constant terms in the denominator, we see that $Q_{12} = \bar{Q}_{12}$. Hence, $Q = \bar{Q}$ and we conclude that the topology of \eqref{system} is identifiable. However, $S$ does not have full column rank and $R$ does not have full row rank.  
\end{example}

\section{Topology identification approach}
\label{sectionresultsidentification}
In this section, we focus on the problem of topology identification, as formulated in Problem~\ref{problem2}. The proposed solution consists of two steps: first identify the Markov parameters of the networked system \eqref{system}, and then extract the matrix $Q$. There are several ways of computing the Markov parameters on the basis of input/output data, we will summarize some of them in the next section.

\subsection{Identification of Markov parameters}

Consider a general linear system of the form 
\begin{align}
x(t+1) &= Ax(t) + Bu(t) \label{syseq1}\\
y(t) &= Cx(t) + Du(t), \label{syseq2}
\end{align}
where $x \in \mathbb{R}^n$ is the state, $u \in \mathbb{R}^m$ is the input and $y \in \mathbb{R}^p$ the output. In this section we recap how one can identify the Markov parameters $D, CB,CAB,\dots,CA^r B$ for $r \in \mathbb{N}$, using measurements of the input and output of \eqref{syseq1}-\eqref{syseq2}. For a given signal $f(t)$ with $t = 0,\dots,T-1$, we define the Hankel matrix of depth $k$ as
\begin{equation*}
\mathcal{H}_k(f) := \begin{pmatrix}
f(0) & f(1) & \cdots & f(T-k) \\
f(1) & f(2) & \cdots & f(T-k+1) \\
\vdots & \vdots &  & \vdots \\
f(k-1) & f(k) & \cdots & f(T-1)
\end{pmatrix}.
\end{equation*}
The signal $f(0),f(1),\dots,f(T-1)$ is said to be \emph{persistently exciting} of order $k$ if $\mathcal{H}_k(f)$ has full row rank. Now suppose that we measure $T$ samples of the input $u(t)$ and output $y(t)$ of \eqref{syseq1}-\eqref{syseq2} for $t = 0,1,\dots,T-1$. We rearrange these measurements in Hankel matrices of depth $n+r+1$. Moreover, we partition  
\begin{align*}
\mathcal{H}_{n+r+1}(u) = \begin{pmatrix}
U_p \\ U_f
\end{pmatrix}, \quad \mathcal{H}_{n+r+1}(y) = \begin{pmatrix}
Y_p \\ Y_f
\end{pmatrix}, 
\end{align*}
where $U_p$ and $Y_p$ contain the first $n$ row blocks of $\mathcal{H}_{n+r+1}(u)$ and $\mathcal{H}_{n+r+1}(y)$, respectively. The following result from \cite[Prop. 4]{Markovsky2008} shows how the Markov parameters can be obtained from data.
\begin{theorem}
	\label{theoremMarkovcomp}
	Let \eqref{syseq1} be controllable and assume that $u(0),\dots,u(T-1)$ is persistently exciting of order $2n+r+1$. There exists a matrix $G\in \mathbb{R}^{(T-n-r) \times m}$ such that
	$$
	\begin{pmatrix}
	U_p \\ Y_p \\ U_f
	\end{pmatrix}G = \begin{pmatrix}
	0 \\ 0 \\ \col(I,0)
	\end{pmatrix}.
	$$
	Moreover, the Markov parameters can be obtained as $Y_f G = \col(D,CB,CAB,\dots,CA^rB)$.
\end{theorem} 

Theorem \ref{theoremMarkovcomp} shows how the Markov parameters of the system can be obtained from measured input/output data. The input should be designed in such a way that it is persistently exciting, special cases of such inputs have been discussed in \cite{Verhaegen1992}.
For $u(0),\dots,u(T-1)$ to be persistently exciting of order $2n+r+1$ a number of samples $T \geq (m+1)(2n+r+1)-1$ is necessary. In fact, there are input functions that achieve persistency of excitation of this order exactly for $T = (m+1)(2n+r+1)-1$. A refinement of Theorem \ref{theoremMarkovcomp} is possible using the notion of weaving trajectories \cite{Markovsky2005}, which reduces 
the order of excitation to $2n+1$. More generally, one can extend the notion of persistency of excitation to an arbitrary concatenation of multiple trajectories \cite{vanWaarde2020LCSS}. This is useful in
situations where single experiments are individually not sufficiently informative.

\begin{remark}
In addition to the deterministic setting of Theorem \ref{theoremMarkovcomp}, there are approaches to identify the Markov parameters of systems with disturbances, i.e., systems of the form
\begin{align*}
x(t+1) &= Ax(t) + Bu(t) + w(t) \\
y(t) &= Cx(t) + Du(t) + v(t),
\end{align*}
where $v$ and $w$ are zero mean, white vector sequences. In particular, the paper \cite{Oymak2018} studies the identification of the system's Markov parameters from finite data, and provides statistical guarantees for the quality of estimation. 
\end{remark}

\subsection{Topology identification}

Subsequently, we will turn to the problem of identifying the topology of \eqref{system} from the network's Markov parameters. As in Theorem \ref{necsufconditions}, we will assume that $S$ has full column rank. In fact, to lighten the notation, we will simply assume $S = I$, even though all results can be stated for general matrices $S$ having full column rank. Under the latter assumption, the Markov parameters of \eqref{system} are given by
$$
M_\ell(Q) := C(A+BQC)^\ell BR.
$$
Whenever the dependence of $M_\ell(Q)$ on $Q$ is clear, we simply write $M_\ell$. It is not immediately clear how to obtain $Q$ from the Markov parameters since $M_\ell$ depends on the $\ell$-th power of $A+BQC$. The following lemma will be helpful since it implies that $M_\ell$ can essentially be viewed as an affine function in $Q$ and lower order Markov parameters.

\begin{lemma}
	\label{lemmaMarkov}
	We have that
	$$
	M_\ell = CA^\ell BR + \sum_{i = 0}^{\ell-1} CA^iBQ M_{\ell-i-1}.
	$$ 	
\end{lemma}

\begin{proof}
	First, we claim that for square matrices $D_1$ and $D_2$ of the same dimensions, we have 
	\begin{equation}
	\label{KLpowers}
	(D_1 + D_2)^\ell = D_1^\ell + \sum_{i=0}^{\ell-1} D_1^i D_2 (D_1+D_2)^{\ell-i-1}
	\end{equation}
	for all $\ell = 1,2,\dots$. It is straightforward to prove this claim by induction. Indeed, for $\ell = 1$, \eqref{KLpowers} holds. If \eqref{KLpowers} holds for $\ell \geq 1$ then 
	\begin{align*}	
	(D_1 + D_2)^{\ell+1} &= D_1^\ell(D_1 + D_2) + \sum_{i=0}^{\ell-1} D_1^i D_2 (D_1 + D_2)^{\ell-i} \\
	&= D_1^{\ell+1} + \sum_{i=0}^{\ell} D_1^i D_2 (D_1 + D_2)^{\ell-i},
	\end{align*}
	proving the claim. Subsequently, by substitution of $D_1 = A$ and $D_2 = BQC$ into \eqref{KLpowers}, we obtain
	\begin{equation*}
	(A + BQC)^\ell = A^\ell + \sum_{i=0}^{\ell-1} A^i BQC (A+BQC)^{\ell-i-1}.
	\end{equation*}
	Finally, the lemma follows by pre- and postmultiplication by $C$ and $BR$, respectively. 
\end{proof}

Using Lemma \ref{lemmaMarkov}, we can come up with a system of linear equations in the unknown interconnection matrix $Q$. To see this, let us denote $K_\ell := M_\ell - CA^\ell BR$. Moreover, define the Toeplitz matrix $L$ by 
$$
L := \begin{pmatrix}  
C B & 0 & \cdots & 0 \\
C A B & CB & \cdots & 0 \\
\vdots & \vdots & \ddots & \vdots \\
C A^{r-1} B & C A^{r-2}B & \cdots & C B
\end{pmatrix},
$$
where $r \geq 2n-1$. We apply Lemma \ref{lemmaMarkov} for $\ell = 1,\dots,r$ to obtain
\begin{equation}
\label{Sylvesterfull}
\begin{pmatrix}
K_1 \\ K_2 \\ \vdots \\ K_r
\end{pmatrix} =
L (I \otimes Q)
\begin{pmatrix}
M_0 \\ M_1 \\ \vdots \\  M_{r-1}
\end{pmatrix}.
\end{equation}
Next, let $L_i$ denote the $(i+1)$-th column block of $L$ and define the matrix $K := \col(K_1,K_2,\dots,K_r)$. We can then write \eqref{Sylvesterfull} in a more compact form as
\begin{equation}
K = \sum_{i=0}^{r-1} L_i Q M_i,
\end{equation}
which reveals that $Q$ is a solution to a \emph{generalized Sylvester equation}. Topology identification thus boils down to i) identifying the network's Markov parameters, ii) constructing the matrices $K$, $L_i$ and $M_i$ for $i = 0,\dots,r-1$ and iii) solving the Sylvester equation. We summarize this procedure in the following theorem. 

\begin{theorem}
	\label{theoremSylvester}
	Consider the networked system \eqref{system} with $S = I$. Let the Markov parameters of \eqref{system} be $M_i$ for $i = 0,1,\dots,r \geq 2n-1$. Let the matrices $K$ and $L_i$ be as before. If the topology of \eqref{system} is identifiable then the interconnection matrix $Q$ is the unique solution to the generalized Sylvester equation
		\begin{equation}
		\label{Sylvester}
		K = \sum_{i=0}^{r-1} L_i \mathbf{Q} M_i
		\end{equation}
		in the unknown $\mathbf{Q}$.
\end{theorem}

\begin{proof}	
	Note that the interconnection matrix $Q$ is a solution to \eqref{Sylvester} by construction. Suppose that $\bar{Q}$ is also a solution to \eqref{Sylvester}. We want to prove that $Q = \bar{Q}$. Since $Q$ and $\bar{Q}$ are both solutions to \eqref{Sylvester}, we have
	\begin{equation}
	\label{sumQsumQbar}
	\sum_{i = 0}^{\ell-1} CA^i B Q M_{\ell-i-1}(Q) = \sum_{i = 0}^{\ell-1} CA^i B \bar{Q} M_{\ell-i-1}(Q)
	\end{equation}
	for $\ell = 1,2,\dots,r$. Here we have written the dependence of $M_{\ell-i-1}$ on $Q$ explicitly, to distinguish between $Q$ and $\bar{Q}$. By Lemma \ref{lemmaMarkov} we have
	\begin{align}
	M_\ell(Q) &= CA^\ell BR + \sum_{i = 0}^{\ell-1} CA^iB Q M_{\ell-i-1}(Q) \label{MQ} \\
	M_\ell(\bar{Q}) &= CA^\ell BR + \sum_{i = 0}^{\ell-1} CA^iB \bar{Q} M_{\ell-i-1}(\bar{Q}). \label{MbarQ} 
	\end{align}
	Clearly, $M_0(Q) = CBR = M_0(\bar{Q})$. In fact, we claim that $M_k(Q) = M_k(\bar{Q})$ for all $k = 0,1,\dots,r$. Suppose on the contrary that there exists an integer $s$ such that $0 < s \leq r$ and $M_s(Q) \neq M_s(\bar{Q})$. We assume without loss of generality that $s$ is the smallest integer for which this is the case. Then $M_k(Q) = M_k(\bar{Q})$ for all $k = 0,1,\dots,s-1$. By combining \eqref{sumQsumQbar} and \eqref{MQ} we obtain 
	\begin{equation}
	M_s(Q) = CA^s BR + \sum_{i = 0}^{s-1} C A^i B \bar{Q} M_{s-i-1}(Q).
	\end{equation}
	By hypothesis $M_k(Q) = M_k(\bar{Q})$ for all $k = 0,1,\dots,s-1$, which yields
	\begin{equation*}
	M_s(Q) = CA^s BR + \sum_{i = 0}^{s-1} C A^i B \bar{Q} M_{s-i-1}(\bar{Q}) = M_s(\bar{Q}),
	\end{equation*}
	using \eqref{MbarQ}. This is a contradiction and we conclude that $M_k(Q) = M_k(\bar{Q})$ for all $k = 0,1,\dots,r$. Since $r \geq 2n-1$ it follows from the Cayley-Hamilton theorem that $M_k(Q) = M_k(\bar{Q})$ for all $k \in \mathbb{N}$. Thus, $F_Q(z) = F_{\bar{Q}}(z)$. Finally, as the topology of \eqref{system} is identifiable, we conclude that $Q = \bar{Q}$. This completes the proof. 
\end{proof}

\subsection{Solving the generalized Sylvester equation}

In the previous section, we saw that the generalized Sylvester equation \eqref{Sylvester} plays a central role in our topology identification approach. In this section, we discuss methods to solve this equation. One simple approach to the problem is to vectorize $\mathbf{Q}$ and write \eqref{Sylvester} as the system of linear equations 
\begin{equation}
\label{Sylvestersystem}
\sum_{i=0}^{r-1} \left( M_i^\top \otimes L_i \right) \vect(\mathbf{Q}) = \vect(K)
\end{equation}
in the unknown $\vect(\mathbf{Q})$ of dimension
$$
\left(\sum_{i=1}^N m_i\right)\left(\sum_{i=1}^N p_i\right).
$$
However, a drawback of this approach is that the dimension of $\vect(\mathbf{Q})$ is quadratic in the number of nodes $N$. This means that for large networks, solving \eqref{Sylvestersystem} is costly from a computational point of view.  

For the `ordinary' Sylvester equation of the form 
$$L_0 \mathbf{Q} + \mathbf{Q} M_1 = K,$$
there are well-known solution methods that avoid vectorization\footnote{It is typically assumed that the matrices $L_0$ and $M_1$ are square \cite{Bartels1972,Golub1979}.}. The general idea is to transform the matrices $L_0$ and $M_1$ to a suitable form so that the Sylvester equation is easier to solve. A classic approach is the \emph{Bartels-Stewart} method \cite{Bartels1972} that transforms $L_0$ and $M_1$ to real Schur form by means of two orthogonal similarity transformations. The resulting equivalent Sylvester equation is then simply solved by backward substitution. A Hessenberg-Schur variant of this algorithm was proposed in \cite{Golub1979}. The approach was also extended to be able to deal with the more general equation
$$ 
L_0 \mathbf{Q} M_0 + L_1 \mathbf{Q} M_1 = K,
$$
using QZ-decompositions \cite[Sec. 7]{Golub1979}. The problem with all of these transformation methods is that they rely on the fact that the Sylvester equation consists of exactly two $\mathbf{Q}$-dependent terms, i.e., $r = 1$. Therefore, it does not seem possible to extend such methods to solve generalized Sylvester equations of the form \eqref{Sylvester} for $r > 1$, see also the discussion in \cite[Sec. 2]{vanLoan2000}. 

Nonetheless, we can improve upon the basic approach of vectorization \eqref{Sylvestersystem} by noting that the matrices $A$, $B$ and $C$ have a special structure. Indeed, recall from \eqref{ABC} that these matrices are block diagonal. This allows us to write down a Sylvester equation for each row block of $\mathbf{Q}$. Let $\mathbf{Q}^{(j)}$ denote the $j$-th block row of $\mathbf{Q}$ for $j \in \mathcal{V}$. Then it is straightforward to show that \eqref{Sylvester} is equivalent to 
\begin{equation}
\label{nodeSylvester}
K^{(j)} = \sum_{i=0}^{r-1} L_{i}^{(j)} \mathbf{Q}^{(j)} M_i
\end{equation}
for all $j \in \mathcal{V}$, where $L_i^{(j)}$ is the $(i+1)$-th column block of the matrix $L^{(j)}$, given by
$$
L^{(j)} := \begin{pmatrix}  
C_j B_j & 0 & \cdots & 0 \\
C_j A_j B_j & C_jB_j & \cdots & 0 \\
\vdots & \vdots & \ddots & \vdots \\
C_j A_j^{r-1} B_j & C_j A_j^{r-2} B_j & \cdots & C_j B_j
\end{pmatrix},
$$
and $K^{(j)} := \col( K_1^{(j)}, K_2^{(j)},\dots,K_r^{(j)} )$ with $K_\ell^{(j)}$ the $j$-th row block of $K_\ell$. The importance of \eqref{nodeSylvester} lies in the fact that each row block of $Q$ can be obtained independently, which significantly reduces the dimensions of the involved matrices. In fact, \eqref{nodeSylvester} is equivalent to the linear system of equations
\begin{equation}
\label{Sylvesternodesystem}
\sum_{i=0}^{r-1} \left( M_i^\top \otimes L_i^{(j)} \right) \vect\left(\mathbf{Q}^{(j)}\right) = \vect\left(K^{(j)}\right)
\end{equation}
in the unknown $\vect\left(\mathbf{Q}^{(j)}\right)$ of dimension $m_j \left(\sum_{i=1}^N p_i\right)$. Note that the unknown is linear in the number of nodes, assuming that $m_j$ and $p_i$ are small in comparison to $N$.

\subsection{Robustness analysis}
In the case that the Markov parameters $M_0,M_1,\dots,M_{r}$ are identified exactly, we can reconstruct the topology by solving the generalized Sylvester equation \eqref{Sylvester}, or equivalently, the system of linear equations \eqref{Sylvestersystem}. 
Now suppose that our estimates of the Markov parameters are inexact, and we have access to 
\begin{equation}
\label{noisyMarkov}
\hat M_\ell := M_\ell + \Delta_\ell, \quad \ell=1,2,\ldots,r
\end{equation}
where the real matrices $\Delta_\ell$
represent the
perturbations. Accordingly, we define $\hat K_\ell := \hat M_\ell - CA^\ell BR = K_\ell + \Delta_\ell$.
Let $\Delta := \col(\Delta_1,\Delta_2,\dots,\Delta_{r})$. In this case it is natural to look for an approximate (least squares) solution $\vect(\hat Q)$ that solves
\begin{equation}
\label{min}
\min_{\vect(\mathbf{\hat Q})} \norm{\sum_{i=0}^{r-1} \left( \hat M_i^\top 
\otimes L_i\right) \vect(\mathbf{\hat Q}) - \vect(\hat K)}.
\end{equation} 
An obvious question is how the solution $\hat{Q}$ is related to the true interconnection matrix $Q$. The following lemma provides a bound on the infinity norm of $\vect(\hat{Q})-\vect(Q)$. In what follows, we will make use of the constant
\begin{align*}
\alpha := \norm{ \left( \sum_{i=0}^{r-1} 
\left( \hat M_i^\top \otimes L_i\right) \right)^\dagger }_{\infty},
\end{align*}
where $X^\dagger$ denotes the Moore-Penrose inverse of $X$.

\begin{lemma}
	Consider the network \eqref{system} with $S=I$ and suppose that its topology be identifiable. Assume that the solution $\hat{Q}$ to \eqref{min} is unique. Then we have that
	$$
	\norm{\vect(\hat{Q})-\vect(Q)}_{\infty}
	$$
    is upper bounded by
	\begin{equation}
	\label{upperbound}
	 \alpha \left(\norm{\vect(\Delta)}_{\infty} + \norm{\sum_{i=0}^{r-1} (\Delta_i^\top \otimes L_i)}_{\infty} \norm{\vect(Q)}_{\infty}\right).
	\end{equation}
\end{lemma}

Note that the bound \eqref{upperbound} tends to zero as $\Delta_0,\Delta_1,\dots,\Delta_r$ tend to zero, so $\hat{Q}$ is a good approximation of $Q$ for small perturbations. An overestimate of \eqref{upperbound} can be obtained if some prior knowledge is available. In particular, note that $\alpha$ is readily computable from the estimated Markov parameters \eqref{noisyMarkov}. The first two norms in \eqref{upperbound} can be upper bounded if a bound on $\norm{\Delta_i}_\infty$ is given. Identification error bounds on the Markov parameters are derived, e.g., in \cite{Oymak2018}. 
Finally, to estimate $\norm{\vect(Q)}_{\infty}$ one requires a bound 
on the largest network weight, i.e., an upper bound on the largest (in magnitude) entry of $Q$. 
The upper bound \eqref{upperbound} is useful in the case that the nonzero weights of the network are lower bounded in magnitude by some known positive scalar $\gamma$, an assumption that is common in the literature on consensus networks, cf. \cite[Sec. 3]{LeBlanc2013}. Indeed, in this case we can can exactly identify the graph structure $\mathcal{G}$ from noisy Markov parameters if 
$$
\alpha \left( \norm{\vect(\Delta)}_{\infty} + \norm{\sum_{i=0}^{r-1} (\Delta_i^\top \otimes L_i)}_{\infty} \norm{\vect(Q)}_{\infty}\right) < \frac{1}{2}\gamma,
$$
since identified entries smaller than $\frac{1}{2}\gamma$ are necessarily zero. We will further illustrate this point in Example \ref{examplenoise}.

\begin{proof}
	We make use of the shorthand notation
	$$
	E := \sum_{i=0}^{r-1} \left(\Delta_i^\top \otimes L_i\right), \: 
	A_E := \sum_{i=0}^{r-1} \left( \hat M_i^\top \otimes L_i\right).
	$$
	The hypothesis that $\hat{Q}$ is unique is equivalent to $A_E$ having full column rank. 
	By using \eqref{Sylvestersystem} and the relation
	$\hat M_i = M_i + \Delta_i$, we get
	$$
	A_E^\top A_E \vect(Q) = A_E^\top (\vect(K)+E\vect(Q)).
	$$
	Therefore, $\vect(Q) = A_E^\dagger (\vect(K)+E\vect(Q))$. Further, 
	$\vect(\hat{Q}) = A_E^\dagger \vect(\hat K) = A_E^\dagger \vect(K+\Delta)$.
	This yields
	\begin{equation*}
	\vect(\hat{Q}) - \vect(Q) = A_E^\dagger(\vect(\Delta)-E\vect(Q)).
	\end{equation*}
	Finally, taking infinity norms yields the upper bound \eqref{upperbound}. This completes the proof.
\end{proof}

\begin{example}
	\label{examplenoise}
	Consider the networked system in Example \ref{example1}. We consider the situation in which only the first node of the network is externally excited. We already know by the discussion in Example \ref{example1} that the topology of the system is identifiable. Here, our aim is to reconstruct the topology on the basis of the noisy Markov parameters \eqref{noisyMarkov}, where $r = 40$. The perturbations are drawn randomly from a normal distribution using the Matlab command {\tt randn}, and scaled such that $\norm{\Delta_i^\top}_{\infty} \leq 10^{-5}$ for all $i$. Since $\Delta_i$ is a vector, this also implies that $\norm{\Delta_i}_\infty \leq 10^{-5}$. In this example, we assume that the weights of the network (i.e., the entries of $Q$) have magnitudes between $\frac{1}{2}$ and $1$.
	
	We identify the matrix $\hat{Q}$ by solving \eqref{min}. To get an idea of the quality of estimation, we want to find a bound on \eqref{upperbound}. First, we compute $\alpha = 464.7040$. By the assumptions on the perturbations and network weights, we obtain the bounds $\norm{\Delta}_{\infty} \leq 10^{-5}$ and $\norm{\vect(Q)}_{\infty} \leq 1$. Moreover, 
	\begin{align*}
	\norm{\sum_{i=0}^{r-1} (\Delta_i^\top \otimes L_i)}_{\infty} &\leq \sum_{i=0}^{r-1} \norm{\Delta_i^\top}_{\infty} \norm{L_i}_{\infty} \\
	&\leq 4.0000 \times 10^{-4},
	\end{align*}
	where we have used \cite[Thm. 8 \& p. 413]{Lancaster1972} to bound the Kronecker product. Combining the previous bounds, we conclude that \eqref{upperbound} is less then or equal to $0.1883$. Since $0.1883 \leq 0.25$ we can round all entries of $\hat{Q}$ that are less than $0.25$ to zero, since the corresponding entries in $Q$ are necessarily zero. The resulting zero/nonzero structure of $\hat{Q}$ can be captured by a graph $\hat{\mathcal{G}}$ that we display in Figure \ref{fig:graph}. Clearly, the structure of $\hat{\mathcal{G}}$ is identical to the graph defined in Example \ref{example1}, and the weights of $\hat{\mathcal{G}}$ are close to the weights of $\mathcal{G}$. Next, we repeat the experiment for larger perturbations, i.e., for  $\norm{\Delta_i}_\infty$ and $\norm{\Delta_i^\top}_\infty$ bounded by $0.01$. We identify $\hat{Q}$ and use the same rounding strategy as before to obtain a graph $\hat{\mathcal{G}}$ in Figure \ref{fig:graph2}. Note that $\hat{\mathcal{G}}$ resembles the original network structure $\mathcal{G}$. In fact, all links are identified correctly, except for $(7,8)$ and the spurious link $(4,8)$. In this case, the bound \eqref{upperbound} equals $49.9997$, illustrating the fact that \eqref{upperbound} can be conservative.
	\begin{figure}[h!]
	\begin{minipage}{0.245\textwidth}
		\centering
			\includegraphics[width=\textwidth,center]{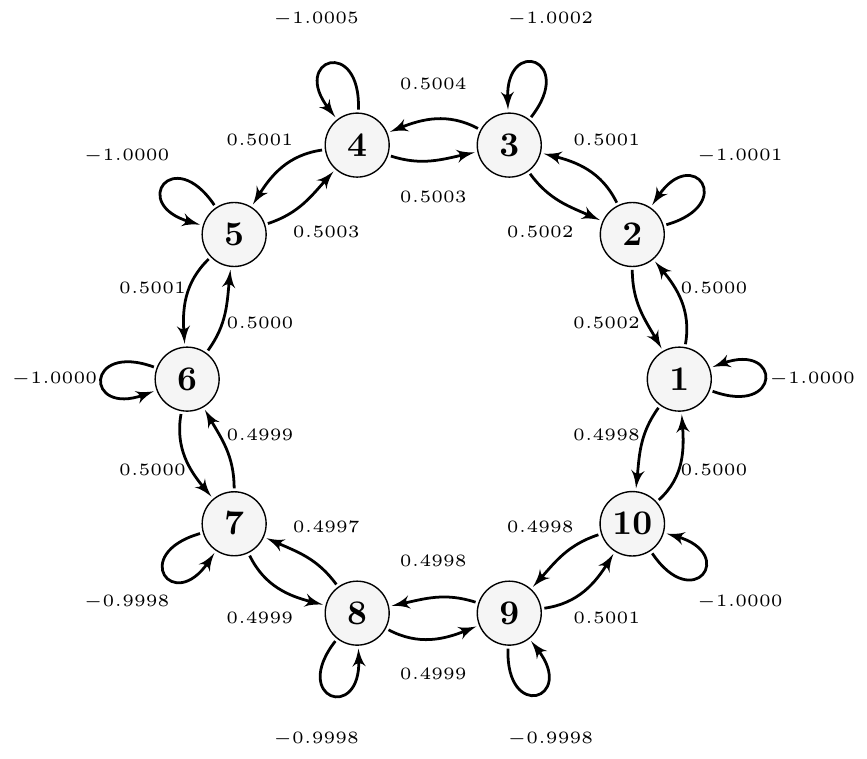}
			\caption{$\hat{\mathcal{G}}$ for $\norm{\Delta_i}_\infty \leq 10^{-5}$.}
			\label{fig:graph}
	\end{minipage}%
	\begin{minipage}{0.245\textwidth}
	\centering
		\includegraphics[width=\textwidth,center]{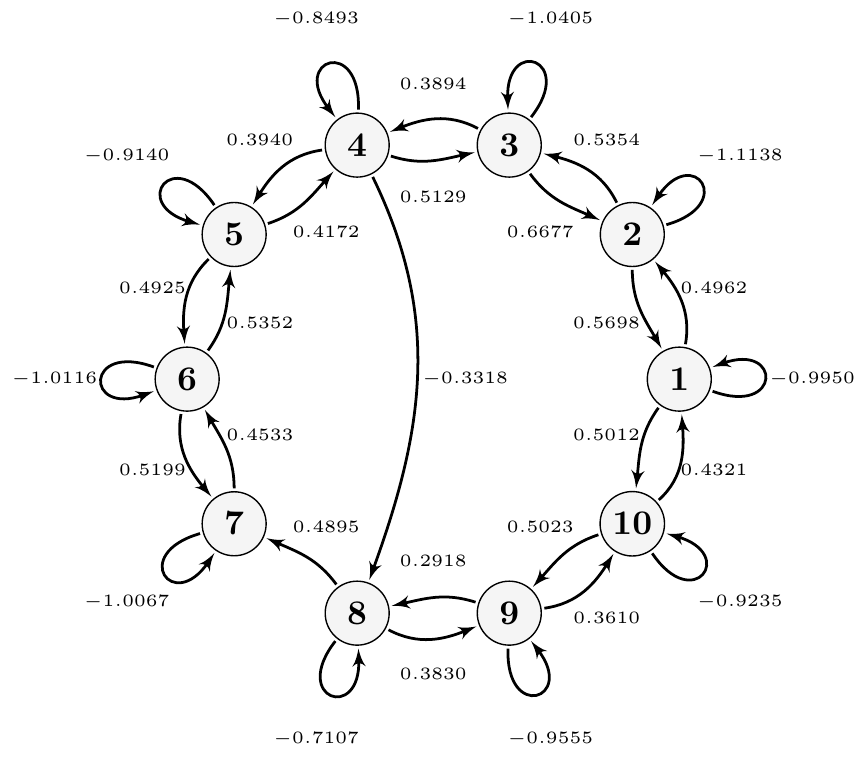}
		\caption{$\hat{\mathcal{G}}$ for $\norm{\Delta_i}_\infty \leq 10^{-2}$.}
		\label{fig:graph2}
	\end{minipage}
\end{figure}
\end{example}
\section{Conclusions}
\label{sectionconclusion}
In this paper we have studied the problem of topology identification of heterogeneous networks of linear systems. First, we have provided necessary and sufficient conditions for topological identifiability. These conditions were stated in terms of the constant kernel of certain network-related transfer matrices. We have also seen that homogeneous SISO networks enjoy quite special identifiability properties that do not extend to the heterogeneous case. Subsequently, we have turned our attention to the topology identification problem. The idea of the identification approach was to solve a generalized Sylvester equation involving the network's Markov parameters to obtain the network topology. One of the attractive features of the approach is that the structure of the networked system can be exploited so that each row block of the interconnection matrix can be obtained individually.  

The generalized Sylvester equation \eqref{Sylvester} plays an important role in our identification approach. Numerical solution methods are less well-developed for this equation than they are for the standard Sylvester equation \cite{Bartels1972,Golub1979}. Hence, it would be of interest to further develop numerical methods for Sylvester equations of the form \eqref{Sylvester}. We note that a Krylov subspace method has already been developed in \cite{Bouhamidi2008}. Another direction for future work is to study topological identifiability with prior information on the interconnection matrix. For example, from physical principles it may be known that $Q$ is Laplacian. Such prior knowledge could be exploited to weaken the conditions for identifiability in Theorems \ref{necsufconditions}, \ref{SRfullrank} and \ref{theoremSISO}.

\bibliographystyle{plainnat}
\bibliography{MyRef}          
                                
\end{document}